\documentclass[10pt,a4paper]{article}
\usepackage[dvipsnames]{xcolor}
\usepackage{amsmath, amsthm, amsfonts, amssymb}
\usepackage[utf8]{inputenc}
\usepackage{enumitem}
\usepackage{tikz}
\usepackage{parskip}
\usepackage{hyperref}

\def\Z{\mathbb{Z}}
\def\R{\mathbb{R}}

\newcommand{\indep}{\perp \!\!\! \perp}
\newcommand{\glb}{\operatorname{global}}
\newcommand{\pa}{\operatorname{Pa}}
\newcommand{\init}{\operatorname{in}}

\theoremstyle{plain}
\newtheorem{thm}{Theorem}[section]
\newtheorem{lem}[thm]{Lemma}

\newtheorem{cor}[thm]{Corollary}
\newtheorem{conj}[thm]{Conjecture}

\theoremstyle{definition}

\newtheorem{exx}[thm]{Example}
\newenvironment{ex}
  {\pushQED{\qed}\exx}
  {\popQED\endexx}

\theoremstyle{remark}
\newtheorem{rmk}[thm]{Remark}
\newtheorem{que}{Question}

\title{Toric and non-toric Bayesian networks}
\author{Lisa Nicklasson\thanks{The author was supported by the grant KAW 2019.0512}}
\date{\small Università degli Studi di Genova\\ \texttt{nicklasson@dima.unige.it}}

\begin{document}
\maketitle

\begin{abstract}
\noindent In this paper we study Bayesian networks from a commutative algebra perspective. We characterize a class of toric Bayesian nets, and provide the first example of a Bayesian net which is provably non-toric under any linear change of variables. Concerning the class of toric Bayesian nets, we study their quadratic relations and prove a conjecture by Garcia, Stillman, and Sturmfels \cite{GSS} for this class. In addition, we give a necessary condition on the underlying directed acyclic graph for when all relations are quadratic. 
\end{abstract}

\section{Introduction}

A \emph{Bayesian network} is a graphical model given by a directed acyclic graph (DAG) where the vertices are random variables. In this paper we consider \emph{finite} Bayesian networks, in the sense that the DAG has a finite number of vertices, and each random variable takes a finite number of values. In general, a \emph{statistical model} is identified with a collection of probability distributions and can often, as in the case of finite Bayesian nets, be realized as a real algebraic variety. Such statistical models can be described explicitly by finding a parameterization of the variety, or implicitly as the zero set of a system of equations. For a Bayesian net $G$ the implicit description gives rise to a polynomial prime ideal which we denote $P_G$. 

Every finite Bayesian network has two graphical representations: the DAG mentioned above, and a \emph{staged tree}. Staged trees are rooted directed trees where a vertex coloring encodes invariances among conditional distributions, and describe a large class of statistical models, not all which are Bayesian networks. As staged trees were first introduced in \cite{SA} earlier work on Bayesian networks does not mention this representation. Nevertheless, staged trees proves to be a helpful tool when studying finite Bayesian networks, not least in the work presented in this paper.

In Section 2 we introduce staged trees, Bayesian networks, and their associated ideals in a fashion that suits the purposes of this paper. 
For a more thorough introduction to staged tree models the reader is referred to \cite{CGS}. An introduction to graphical models can be found for instance in \cite{Lauritzen}. 

\emph{Toric models} are statistical models which can be realized as toric varieties. Clearly, being toric opens up to using tools for from the well studied research field of toric geometry, and its benefits for statistical models are discussed for example in \cite{Rapallo}.
An example of a class of toric models are the discrete \emph{undirected graphical models} \cite{GMS}. Related to directed graphical models, the so called  
\emph{Conjunctive Bayesian networks} studied in \cite{BES} are toric. A class of toric staged trees is characterized in \cite{GMN}. In Section 3 of this paper we investigate which finite Bayesian networks are toric. Applying the result of \cite{GMN} we deduce a class of toric Bayesian nets, described by a property on the induced subgraph of the non-sinks of the DAG, see Theorem \ref{thm:toric_BN}. In general it is a difficult task to prove that a variety is not toric, under any linear change of coordinates. However, in Theorem \ref{thm:non-toric} we give the first example of a Bayesian network proved to not be toric in any choice of basis. 

An implicit description of a graphical model is often given in terms of conditional independence relations. Given a finite Bayesian network $G$, these relations gives rise to a quadratic ideal $I_G \subseteq P_G$. In Section 4 we examine the relationship between these two ideals. It is conjectured in \cite{GSS} that the quadratic part of $P_G$ is given precisely by $I_G$. In Theorem \ref{thm:global} we prove this conjecture for the toric Bayesian nets given by Theorem \ref{thm:toric_BN}. Moreover, we show in Theorem \ref{thm:quadratic} that, under the assumption of Theorem \ref{thm:toric_BN}, for the two ideals to be equal it is necessary that $G$ does not contain an induced cycle of length greater than three. 

In Section 5 we state three open problems, hoping to inspire further research on commutative algebra of Bayesian networks. 

\section{Preliminaries}
\subsection{Staged trees}
Consider a directed tree $T=(V,E)$ with a distinguished root, such that every edge is directed away from the root. For a vertex $v \in V$ let $E(v) \subset E$ denote the set of edges $v \to v'$ emanating from $v$. To each edge we assign a label, and we let $\Theta$ denote the set of edge labels. The function $\theta: E \to \Theta$ maps each edge to its edge label. The tree $T$ is called a \emph{staged tree} if 
\begin{itemize}
\item for every $v \in V$ the edges $E(v)$ gets distinct labels, i.\,e.\ $|\theta(E(v))|=|E(v)|$,
\item for every pair $v, w \in V$ the sets $\theta(E(v))$ and $\theta(E(w))$ are either equal or disjoint.  
\end{itemize}
The tree to the right in Figure \ref{fig:first_example} is an example of a staged tree. 
We define an equivalence relation $\sim$ on $V$ by $v \sim w$ if $\theta(E(v))=\theta(E(w))$. The equivalence classes are called \emph{stages}.
When drawing a staged tree $T$ we visualize the stages with a vertex coloring by giving vertices of the same stage the same color. To make the pictures in this paper as clear as possible, we choose to only give color to stages with more than one vertex. White vertices should be considered uncolored, meaning that each white vertex belongs to a stage consisting of only that single vertex. 

The \emph{level} of a vertex $v$ is the distance, i.\,e.\ the number of edges, between $v$ and the root. A staged tree is called \emph{stratified} if vertices of the same stage are on the same level, and all leaves are on the same level. All staged trees considered in this paper are stratified. 

Let $\rho: \Theta \to (0,1)$ be a map that assigns a real value in the open interval $(0,1)$ to each edge label, requiring 
\begin{equation}\label{eq:rho}
\sum_{e \in E(v)}\!\! \rho(\theta(e)) =1 \quad \text{for each} \ v \in V.
\end{equation}
Let $N$ be the number of leaves in the tree $T$. Each leaf $\ell$ is assigned a value in $(0,1)$ by taking the product $\prod \rho(\theta(e))$ over all edges $e$ on the directed path from the root to $\ell$. This produces a point in $\R^N$. The set of all points in $\R^N$ obtained by varying $\rho$ is the \emph{staged tree model} $\mathcal{M}_T$.

The staged tree model can also be described in terms of algebraic geometry. Let $\R[\boldsymbol{x}]=\R[x_1, \ldots, x_N]$ be the polynomial ring where each variable is associated to a leaf in the staged tree $T$, and let $\R[\Theta]$ be the polynomial ring on the edge labels. Let $\langle \boldsymbol{\theta} -1 \rangle$ denote the ideal of $\R[\Theta]$ generated by the relations
\[
\sum_{e \in E(v)}\!\!\! \theta(e) -1 \quad \text{for each} \ v \in V.
\]
Note that it is enough to consider one vertex from each stage. Now we define a homomorphism $\varphi: \R[\boldsymbol{x}] \to \R[\Theta]/\langle \boldsymbol{\theta} -1 \rangle$ by sending each variable to the product of the edge labels along the root-to-leaf path. 
Then
\[
\mathcal{M}_T = \mathcal{V}(\ker \varphi) \cap \Big\{ p \in \R^N \ \Big| \ \sum_{i=1}^N p_i =1 \ \text{and} \ 0<p_i<1 \ \text{for} \ i=1, \ldots, N \Big\}.
\]
Note that the ideal $\ker \varphi$ is not homogeneous as $\varphi(x_1+ \dots + x_N -1)=0$. For practical reasons we will work with a homogenized version of the map $\varphi$. To this end we introduce a homogenizing variable $z$, and let $\langle \boldsymbol{\theta} -z \rangle \subset \R[\Theta,z]$ denote the ideal generated by
\[
\sum_{e \in E(v)}\!\!\! \theta(e) -z \quad \text{for each} \ v \in V.
\]
For a given stratified staged tree $T$ we let $P_T$ denote the kernel of the map $\bar \varphi: \R[\boldsymbol{x}] \to \R[\Theta,z]/\langle \boldsymbol{\theta} -z \rangle$, defined in the same way as $\varphi$. Note that $P_T$ is a prime ideal as $\R[\Theta,z]/\langle \boldsymbol{\theta} -z \rangle$ is a domain. Moreover, $P_T$ is homogeneous as $\R[\Theta,z]/\langle \boldsymbol{\theta} -z \rangle$ is a graded ring, and $\bar \varphi$ maps homogeneous polynomials to homogeneous elements. 
 
 \begin{lem}
 With notation as above, and $T$ a stratified staged tree 
 \[\ker \varphi = P_T + \langle x_1 + \dots + x_N -1 \rangle.\]
 \end{lem}
 \begin{proof}
 Let $d$ be the length of a root-to-leaf path in $T$. Take $f \in \ker \varphi$ and consider the image of $f$ under $\bar \varphi$. Each term in $\bar \varphi (f)$ has degree divisible by $d$. We reduce the number of terms in the presentation of $\bar \varphi (f)$ as much as possible working modulo $\langle \boldsymbol{\theta} -z \rangle$. The resulting representation is a polynomial where all terms cancel after substituting $z=1$. In other words, if we see a term $m$ in $\bar \varphi(f)$, there is also a term $-z^{ds}m$. It follows that
 \[ \ker \varphi = \ker \bar \varphi + \bar \varphi^{-1}(\langle z^d-1 \rangle) = \ker \bar \varphi + \langle x_1 + \dots +x_N -1 \rangle.\qedhere \] 
 \end{proof}
 See also the discussion about passing to projective space when studying statistical models for discrete random variables in \cite[Section 3.6]{Sullivant}.

\subsection{Bayesian networks}
Throughout this paper a Bayesian network is a DAG $G=(V,E)$ where the vertex set $V$ is a set of $n$ finite random variables. Each random variable $X$ takes a number of values say $1, \ldots, \kappa$ with some probabilities $\theta_1, \ldots, \theta_\kappa$, given the values of the parents of $X$.

\begin{figure}
\centering
\includegraphics[scale=1.1]{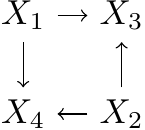}\!\!\!
\includegraphics[scale=0.97]{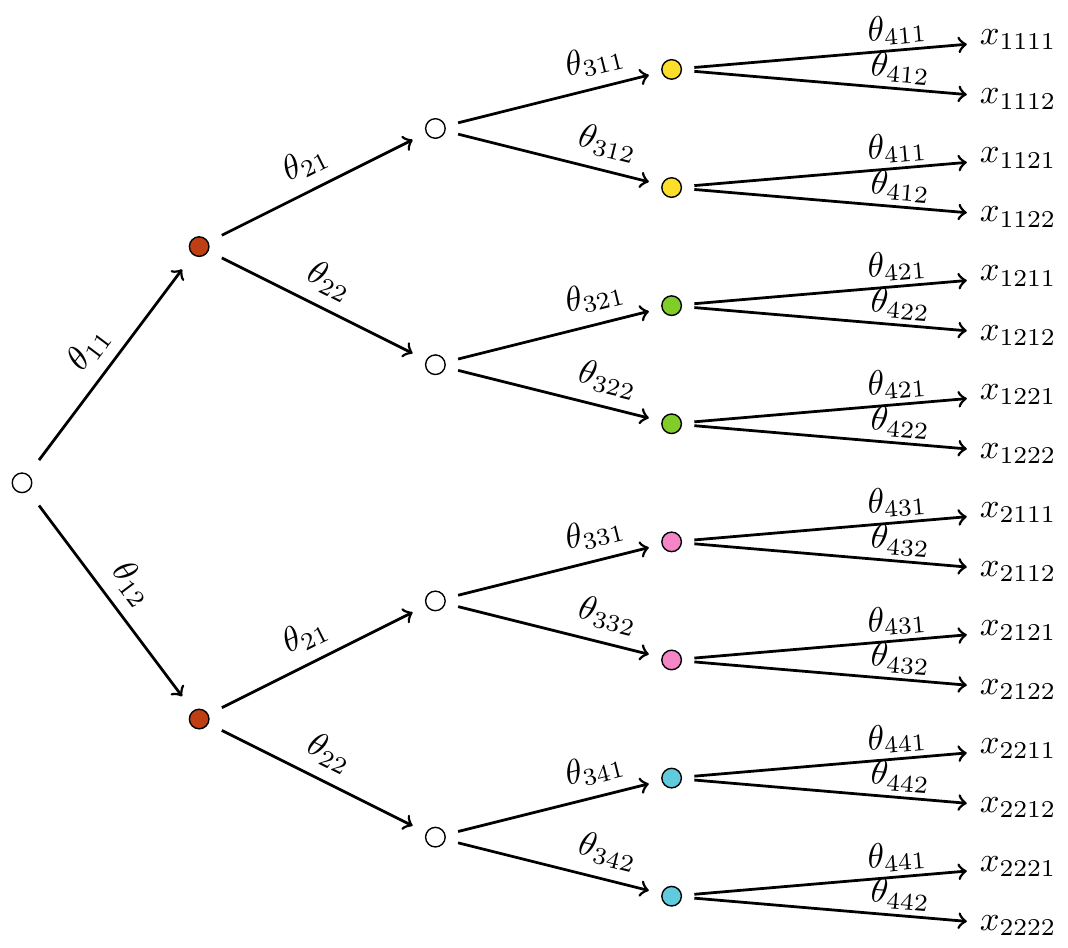}
\caption{A DAG and a staged tree representing a Bayesian network on four binary random variables}
\label{fig:first_example}
\end{figure}

 The Bayesian network $G$ can be represented by a stratified staged tree $T_G$ with leaves on level $n$, as we see for instance in Figure \ref{fig:first_example}. The staged tree is produced in the following way.  
 First we order the vertices $V=\{X_1, \ldots, X_n\}$ in a way such that if $X_i \to X_j$ is an edge, then $i<j$. Let $1, \ldots, \kappa_i$ denote the possible values of the finite random variable $X_i$. The vertices on level $i$ in $T_G$ have $\kappa_{i+1}$ outgoing edges, representing the possible values of $X_{i+1}$. A vertex $v$ on level $j$ can be identified with a vector $v=(v_1, \ldots, v_j)$ where $v_1, \ldots, v_j$ are the values of $X_1, \ldots, X_j$ defining the directed path from the root to $v$. Two vertices $v$ and $w$ on level $j$ are in the same stage if $v_i=w_i$ whenever $X_i \to X_j$ is an edge in $G$. 
When working with staged trees of Bayesian networks´ it is sometimes convenient to employ the notation 
 \[
 \theta(X_j=1 \ | \ X_i=v_i \ \text{for} \ X_i \in \pa(X_j)), \ldots, 
 \theta(X_j=\kappa_j \ | \ X_i=v_i \ \text{for} \ X_i \in \pa(X_j))
 \]
 for the labels of the edges $E(v)$. Here Pa$(X_j)$ stands for the set of parents of $X_j$ in $G$, that is the $X_i$'s such that $X_i \to X_j$ is an edge.  For example, if $v$ is any of the yellow vertices in Figure \ref{fig:first_example}, then the edge labels in the picture translates to the new notation as 
 \[
 \theta_{411} = \theta(X_4=1 \ | \ X_1=1, X_2=1), \quad \theta_{412} = \theta(X_4=2 \ | \ X_1=1, X_2=1).
 \]

We use the notation $P_G$, rather than $P_{T_G}$, for the prime ideal associated to $T_G$. When studying the ideal $P_G$ of a Bayesian network $G$ it can be convenient to index the variables in $\R[\boldsymbol{x}]$ by vectors $u=(u_1, \ldots, u_n)$ encoding the values of $X_1, \ldots, X_n$. Then $P_G$ is defined as the kernel of the map $\R[\boldsymbol{x}] \to \R[\Theta,z]/\langle \boldsymbol{\theta} -z \rangle$ defined by
\begin{equation}\label{eq:map_BN}
x_u \mapsto \prod_{j=1}^n \theta(X_j=u_j \ | \ X_i=u_i \ \text{for} \ X_i \in \pa(X_j)).
\end{equation}
\begin{rmk}\label{rmk:numbering}
The ideal $P_G$ is independent (up to a reindexing of the variables in $\R[\boldsymbol{x}]$) of the choice of ordering of the random variables $X_1, \ldots, X_n$, as long as the numbering respects the directions of the edges. This can be seen from \eqref{eq:map_BN}, as an admissible reordering of $X_1, \ldots, X_n$ only reorders the factors in the image. 
\end{rmk}

 Moreover, we allow replacing some of the $u_i$'s by the symbol + to denote the sum of all $x_u$'s with fixed values for a subset of the random variables $X_1, \ldots, X_n$. More precisely, let $\Lambda \subseteq \{1, \ldots, n\}$, and fix some values $1 \le u_i \le \kappa_i$ for each $i \in \Lambda$. Let $u+$ denote the $n$-vector where the $i$-th entry is $u_i$ if $i \in \Lambda$ and + otherwise. Then $x_{u+} = \sum x_v$ where the sum is taken over all integer vectors $v=(v_1, \ldots, v_n)$ such that $1 \le v_i \le \kappa_i$ and $v_i = u_i$ if $i \in \Lambda$.

When discussing the underlying DAG of a 
Bayesian network we will often denote the vertices by the integers $1, \ldots, n$ referring to the random variables $X_1, \ldots, X_n$. 

\subsection{Conditional independence}\label{subsec:ci-rel}
Let $G$ be a Bayesian network on $n$ vertices, and let $A, B,$ and $C$ be disjoint subsets of the vertices. We write $A \indep B \ | \ C$ for the conditional independence statement ``$A$ is independent of $B$ given $C$''. This should be understood as ``the probability of the random variables in $A$ taking any fixed values is independent of the values of the random variables in $B$, given the values of the random variables $C$''. The perhaps most intuitive conditional independence statements are those of the \emph{ordered Markov property}. The ordered Markov property is the set of conditional independence statements  
\[
j \indep (\{1, \ldots, j-1\} \setminus \pa(j)) \ | \ \pa(j).
\]
For example, the statement $4 \indep 3 \ | \ \{1,2\}$ holds for the DAG in Figure \ref{fig:first_example}.

Suppose we fix a value $1 \le a_i \le \kappa_i$ for each $i \in A$, and let $a$ denote this choice of values. In the same way we fix values $b$ and $c$ for the random variables in $B$ and $C$. Let $abc$ denote the $n$-vector where the $i$-th entry is $a_i, b_i,$ or $c_i$ if $i \in A, B$ or $C$, and + otherwise. Let $M_c$ denote a matrix where the rows are indexed by all combinations $a$ of values for the random variables in $A$, and the column by the different choices of $b$. The entry on position $(a,b)$ in $M_c$ is $x_{abc}$, and we define $I_c$ as the ideal generated by the $2 \times 2$-minors of $M_c$. In other words, the generators of $I_c \subset \R[\boldsymbol{x}]$ are quadratic forms
\begin{equation}\label{eq:indep_eq}
x_{abc}x_{a'b'c}-x_{a'bc}x_{ab'c}.
\end{equation}
We define the ideal $I_{A \indep B | C}$ as the sum of all $I_c$, for all different combinations $c$ of values of the random variables in $C$. 
The conditional independence statement $A \indep B \ | \ C$ is then equivalent to the ideal containment $I_{A \indep B | C}\subseteq P_G$. For a proof of this fact see \cite[Proposition 8.1]{Sturmfels_solving}. 

 A \emph{trail} in a DAG is a path where the directions of the edges are not taken into account. 
Following \cite{Lauritzen} the statement $A \indep B \ | \ C$ translates to graph theoretical terms as $A$ and $B$ being \emph{separated} by $C$ in the following sense. If there is a trail $\pi$ connecting a vertex from $A$ and a vertex from $B$ there must be a vertex $j$ on $\pi$ such that either
\begin{enumerate}
\item[\emph{S1.}] $j \in C$, and the directions of the edges of $\pi$ at $j$ are \emph{not} $\rightarrow j \leftarrow$, or 
\item[\emph{S2.}] $j \notin C$ and $j$ has no descendant in $C$, and the edges of $\pi$ at $j$ are directed as $\rightarrow j \leftarrow$.
\end{enumerate}
In particular, a vertex in $A$ can not have a child or a parent in $B$. 
The set of all conditional independence statements that holds for a Bayesian network $G$ is called the \emph{global Markov property} of $G$. Employing the same notation as \cite{GSS} we let $I_{\glb(G)}$ denote the sum of all ideals $I_{A \indep B | C}$ for which $A \indep B \ | \ C$ holds. 

By definition $I_{\glb(G)} \subseteq P_G$, and it follows by \cite[Theorem 8]{GSS} that $P_G$ is a minimal prime of $I_{\glb(G)}$.

\section{Conditions for toric Bayesian networks}

A \emph{toric ideal} in $\R[\boldsymbol{x}]$ is a prime binomial ideal. Equivalently, an ideal is toric if it is the kernel of a monomial map, i.\,e.\ a homomorphism between polynomial rings where each variable is mapped to a monomial. Note that the ring homomorphisms defining the ideals $P_G$ are not monomial maps, as the image lies in a quotient ring. Moreover, being a binomial ideal is a property that depends on the choice of basis for the polynomial ring $\R[\boldsymbol{x}]$. In this section we study the question: For which Bayesian networks $G$ are the ideals $P_G$ toric, after a suitable linear change of variables?

The most immediate class of toric Bayesian nets are those for which the associated prime ideal is binomial in the given variables. 
A DAG $G$ is called \emph{perfect} if for every vertex $i$ the induced undirected subgraph on Pa$(i)$ is a complete graph. 
It was first proved in \cite[Proposition 3.28]{Lauritzen} that the ideal $P_G$ is binomial when $G$ is perfect. Moreover, by \cite[Theorem 3.1]{Duarte-Solus} $G$ is perfect precisely when the staged tree $T_G$ is a so called \emph{balanced tree}, which in turn is equivalent to the associated prime ideal being binomial in the given variables, \cite{Duarte-Goergen}. We summarize this result in Theorem \ref{thm:perfect}. 

\begin{thm}[\cite{Duarte-Goergen}, \cite{Duarte-Solus}, \cite{Lauritzen}]\label{thm:perfect}
Let $G$ be a Bayesian network. The ideal $P_G$ is binomial in the given variables if and only if $G$ is perfect.
\end{thm}

In \cite{GMN} the class of toric staged trees is extended by considering a change of variables. We shall apply this result to obtain a class of toric Bayesian nets. To do this we restate the special case of \cite[Theorem 5.4]{GMN} concerning stratified staged trees as Theorem \ref{thm:toric_staged_trees}. A \emph{subtree} of a staged tree $T$ always inherits the edge labels from $T$. For two subtrees to be \emph{identical} their edge labelings must be identical. For a vertex $v$ in a staged tree $T$, the \emph{induced subtree of $v$} is the subtree $T(v)$ containing every directed path starting in $v$. 

\begin{thm}[\cite{GMN}]\label{thm:toric_staged_trees}
For a stratified staged tree $T$ and an integer $d \ge 0$, let $S$ be the subtree with same root as $T$ and with leaves on level $d$. Suppose
\begin{enumerate}
\item $S$ is balanced, and
\item for any vertex $v$ on level $\ge d$ the induced subtrees of the children of $v$ are identical.
\end{enumerate}
Then the prime ideal associated to $T$ is toric after a linear change of variables. 
\end{thm}

Let us analyze the second condition. Take a vertex $v$ on level $d$, and let $v_1$ and $v_2$ be children of $v$. Then the induced subtrees $T(v_1)$ and $T(v_2)$ are identical, so in particular $v_1 \sim v_2$. But condition {\it 2} should also hold for $v_1$, so all children of $v_1$ must be in the same stage. As $T(v_1)$ and $T(v_2)$ are identical the $k$-th children of $v_1$ and $v_2$ must be in the same stage. Then in fact all grandchildren of $v$ must be in the same stage. Continuing this argument we can rephrase condition {\it 2} as
\begin{enumerate}
\item[{\it 2'.}]For any vertex $u$ on level $d$ and any vertices $v, w$ in $T(u)$ on the same level, $v \sim w$.  
\end{enumerate}
 
Theorem \ref{thm:toric_staged_trees} translates to the following statement on the DAG $G$, in the case $T=T_G$.

\begin{thm}\label{thm:toric_BN}
Let $G$ be a Bayesian network for which the induced subgraph on the non-sinks is perfect. 
Then the ideal $P_G$ is toric, after a linear change of variables. 
\end{thm}  
\begin{proof}
Suppose the induced subgraph on the non-sinks of $G$ if perfect. We may assume that the vertices are ordered so that the induced subgraph $H$ on the vertices $1, \ldots, d$ is perfect, and that every vertex $i >d$ is a sink. Then condition {\it 1} of Theorem \ref{thm:toric_staged_trees} is satisfied with $S=T_H$. Let now $u$ be a vertex on level $d$ in the staged tree $T_G$, and take two vertices $v$, $w$ on level $j$ both contained in $T(u)$. As $\pa(j) \subseteq \{1, \ldots, d\}$ it follows that $v \sim w$. 
\end{proof}

\begin{rmk}\label{rmk:toric}
The algorithm in the proof of \cite[Theorem 5.4]{GMN} produces the change of variables which makes the ideal $P_G$ binomial. In the case of Theorem \ref{thm:toric_BN} we get the following. For a given integer vector $u=(u_1, \ldots, u_n)$ with $1 \le u_i \le \kappa_i$ let $u+$ be the vector obtained from $u$ by replacing $u_i$ by the symbol + if $i$ is a sink and $u_i=\kappa_i$. The set of all linear forms $x_{u+}$ obtained in this way is the basis of $\R[\boldsymbol{x}]$ for which $P_G$ is binomial. The image of $x_{u+}$ is then obtained by applying the substitution 
\[
\theta(X_j=\kappa_j \ | \ X_i = u_i \ \text{for} \ i \in \pa(j)) \mapsto z \quad \text{if} \ j \ \text{is a sink}
\]
to \eqref{eq:map_BN}. This gives the parametrization of the toric variety $\mathcal{V}(P_G)$.   
\end{rmk}

\begin{ex}
Let $G$ be the Bayesian network of four binary random variables given in Figure \ref{fig:first_example}. The DAG is not perfect, as Pa$(3)=\{1,2\}$ but there is no edge between 1 and 2. But if we remove the two sinks 3 and 4 the resulting graph is perfect, so $P_G$ is toric by Theorem \ref{thm:toric_BN}. The basis for $\R[\boldsymbol{x}]$ and the monomial parameterization is given by
\[
\begin{aligned}
x_{1111} & \mapsto \theta_{11}\theta_{21}\theta_{311}\theta_{411} \\
x_{111+} & \mapsto \theta_{11}\theta_{21}\theta_{311}z \\
x_{11+1} & \mapsto \theta_{11}\theta_{21}z\theta_{411} \\
x_{11++} & \mapsto \theta_{11}\theta_{21}z^2 \\
x_{1211} & \mapsto \theta_{11}\theta_{22}\theta_{321}\theta_{421} \\
x_{121+} & \mapsto \theta_{11}\theta_{22}\theta_{321}z \\
x_{12+1} & \mapsto \theta_{11}\theta_{22}z\theta_{421} \\
x_{12++} & \mapsto \theta_{11}\theta_{22}z^2
\end{aligned}
\qquad \quad
\begin{aligned}
x_{2111} & \mapsto \theta_{12}\theta_{21}\theta_{331}\theta_{431} \\
x_{211+} & \mapsto \theta_{12}\theta_{21}\theta_{331}z \\
x_{21+1} & \mapsto \theta_{12}\theta_{21}z\theta_{431} \\
x_{21++} & \mapsto \theta_{12}\theta_{21}z^2 \\
x_{2211} & \mapsto \theta_{12}\theta_{22}\theta_{341}\theta_{441} \\
x_{221+} & \mapsto \theta_{12}\theta_{22}\theta_{341}z \\
x_{22+1} & \mapsto \theta_{12}\theta_{22}z\theta_{441} \\
x_{22++} & \mapsto \theta_{12}\theta_{22}z^2.
\end{aligned} 
\vspace{-20pt}
\]
\end{ex}

Now the question is: Are there more toric Bayesian networks, not characterized by Theorem \ref{thm:toric_BN}? In \cite[Conjecture 7.1]{GMN} it is conjectured that all Bayesian networks are toric. However, Theorem \ref{thm:non-toric} provides a counterexample to the conjecture. 

\begin{figure}
\centering
\includegraphics[scale=1.3]{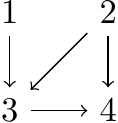}
\hspace{2cm}
\includegraphics[scale=0.9]{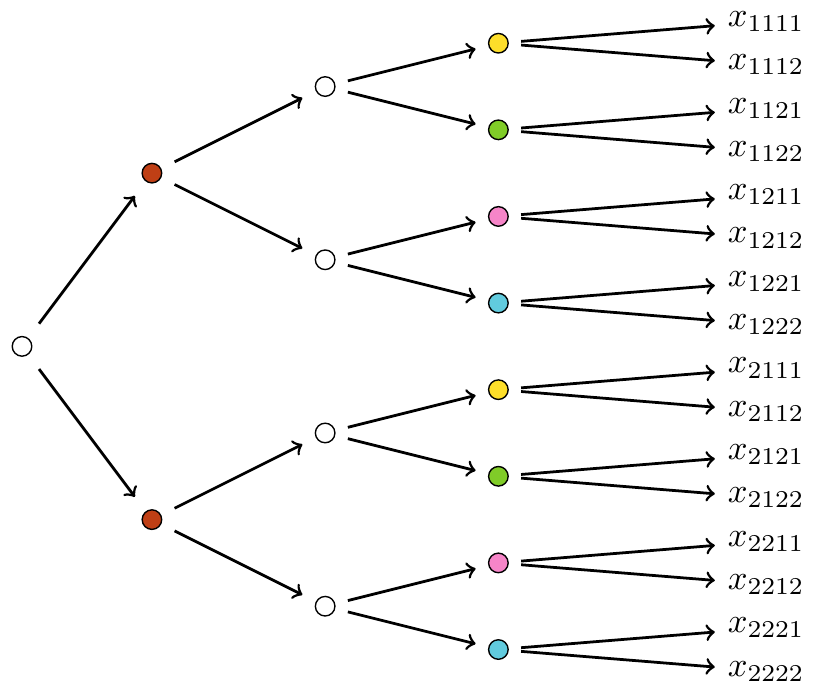}
\caption{DAG and staged tree representation of a non-toric Bayesian network}
\label{fig:non-toric}
\end{figure}

\begin{thm}\label{thm:non-toric}
Let $G$ be the Bayesian network on four binary random variables illustrated in Figure \ref{fig:non-toric}. Then there is no linear change of basis for which the prime ideal $P_G$ is toric.  
\end{thm}
\begin{rmk}
The ideal $P_G$ is toric if and only if the quotient ring $\R[\boldsymbol{x}]/P_G$ is isomorphic to a subring $S=\R[m_1, \ldots, m_s]$ of a polynomial ring, where $m_1, \ldots, m_s$ is a minimal generating set of monomials. We know that $\R[\boldsymbol{x}]/P_G$ is indeed isomorphic to a subring $A \subseteq \R[\Theta,z]$ minimally generated by 16 polynomials. Explicit expressions of the generators are given in Example \ref{ex:normal}, and one easily checks that they are linearly independent. Finding a change of variables under which $P_G$ is binomial is equivalent to finding an isomorphism $\psi: S \to A$. Such an isomorphism $\psi$ sends the given generating set of $S$ to a generating set of $A$. It is therefore necessary that $s=16$, and for each $i$ the image $\psi(m_i)$ must be a linear combination of the generators of $A$. Hence it is enough to consider \emph{linear} changes of variables in order to determine whether $P_G$ is toric or not.   
\end{rmk}
The proof of Theorem \ref{thm:non-toric} uses matrix representation of quadratic forms. Every quadratic form $f$ on variables $x_1, \ldots x_N$ can be represented by a symmetric coefficient matrix $S$ by $f=x^\top S x$. Here $x$ denotes the column vector $(x_1, \ldots, x_N)$. 
The \emph{rank of $f$} refers to the rank of $S$, which is invariant under linear change of coordinates. If $f$ is a binomial then the rank is at most four. 

We also recall some basic facts about minimal generating sets of homogeneous ideals of polynomial rings. A generating set is \emph{minimal} if no proper subset generates the same ideal. Minimal generating sets are not unique in general. However, if $f_1, \ldots, f_r$ and $g_1, \ldots, g_s$ are two minimal homogeneous generating sets for the same ideal, then $r=s$ and (after possibly reindexing) $\deg(f_i)=\deg(g_i)$.

\begin{proof}[Proof of Theorem \ref{thm:non-toric}]
The computational software Macaulay2 \cite{M2} will be used in two steps of this proof. 

The global Markov property of $G$ consists of $2 \indep 1$ and $4 \indep 1 \ | \ \{2,3\}$. The ideal $I_{2 \indep 1}$ is generated by the determinant
\[
f_1= \begin{vmatrix}
x_{11++} & x_{21++} \\
x_{12++} & x_{22++}
\end{vmatrix},
\]
and $I_{4 \indep 1 | \{2,3\}}$ is generated by the four determinants
\begin{align*}
&f_2=\begin{vmatrix}
x_{1111} & x_{2111} \\
x_{1112} & x_{2112}
\end{vmatrix}, \quad f_3=\begin{vmatrix}
x_{1121} & x_{2121} \\
x_{1122} & x_{2122}
\end{vmatrix}, \\
 &f_4=\begin{vmatrix}
x_{1211} & x_{2211} \\
x_{1212} & x_{2212}
\end{vmatrix}, \quad f_5=\begin{vmatrix}
x_{1221} & x_{2221} \\
x_{1222} & x_{2222}
\end{vmatrix}.
\end{align*}
The Macaulay2 commands
\begin{verbatim}
I=ideal(f_1, f_2, f_3, f_4, f_5)
trim I
isPrime I
\end{verbatim}
with $f_1, \ldots, f_5$ being the polynomials given above, verifies that this is a minimal generating set for the ideal $I_{\glb(G)}=I_{2 \indep 1} + I_{4 \indep 1 | \{2,3\}}$. Moreover, the last line tells us that $I_{\glb(G)}$ is a prime ideal. As $P_G$ is a minimal prime of $I_{\glb(G)}$ we have $P_G=I_{\glb(G)}$. 
Hence any minimal generating set of $P_G$ consists of five linear combinations $\lambda_1f_1 + \dots + \lambda_5f_5$. Such a quadratic form is represented by the symmetric matrix 
\[
S=\begin{pmatrix}
0 & A \\
A^\top & 0 
\end{pmatrix}
\quad \text{where} \quad A=\left(\begin{smallmatrix}
 0 & \lambda_{2} & 0 & 0 & \lambda_{1} & \lambda_{1} & \lambda_{1} & \lambda_{1} \\
-\lambda_{2}&0&0&0&\lambda_{1}&\lambda_{1}&\lambda_{1}&\lambda_{1}\\
0&0&0&\lambda_{3}&\lambda_{1}&\lambda_{1}&\lambda_{1}&\lambda_{1}\\
0&0&-\lambda_{3}&0&\lambda_{1}&\lambda_{1}&\lambda_{1}&\lambda_{1}\\
-\lambda_{1}&-\lambda_{1}&-\lambda_{1}&-\lambda_{1}&0&\lambda_{4}&0&0\\
-\lambda_{1}&-\lambda_{1}&-\lambda_{1}&-\lambda_{1}&-\lambda_{4}&0&0&0\\
-\lambda_{1}&-\lambda_{1}&-\lambda_{1}&-\lambda_{1}&0&0&0&\lambda_{5}\\
-\lambda_{1}&-\lambda_{1}&-\lambda_{1}&-\lambda_{1}&0&0&-\lambda_{5}&0\\
 \end{smallmatrix}\right)
\]
in the basis $\{ x_{1111}, x_{1112}, \ldots, x_{2222}\}$.
For $S$ to represent a binomial the rank must be at most four, so all ${5 \times 5}$-minors must vanish. Running the commands
\begin{verbatim}
J = radical minors(5,S)
primaryDecomposition J
\end{verbatim}
where $S$ is the matrix given above with entries in a polynomial ring, computes the radical of the ideal  of $5 \times 5$-minors as 
\[
 (\lambda_1 \lambda_2,\lambda_1 \lambda_3,\lambda_1 \lambda_4,\lambda_1 \lambda_5,\lambda_2 \lambda_3, \lambda_2 \lambda_4,\lambda_2 \lambda_5,\lambda_3 \lambda_4,\lambda_3 \lambda_5, \lambda_4 \lambda_5 ) \]
 which has the primary decomposition
 \[
(\lambda_2,\lambda_3,\lambda_4,\lambda_5) \cap (\lambda_1,\lambda_3,\lambda_4,\lambda_5) \cap \dots \cap (\lambda_1,\lambda_2,\lambda_3,\lambda_4).
\]
This shows that $S$ has rank less than five only if all but one $\lambda_i=0$, so the only possible binomials are $f_1, \ldots, f_5$ themselves. Indeed, $f_2, \ldots, f_5$ are binomials, and $f_1$ is a binomial after a change of variables. The next step is to see that there is no change of variables for which they are all binomials. Any quadratic binomial is the determinant of a $2 \times 2$-matrix where the entries are variables. The fact that $f_1, \ldots, f_5$ are all of rank four means that we are looking for $2 \times 2$-matrices with four distinct entries. Then there must be a pair $f_i$, $f_j$ which share a variable. For such a pair, $f_i+f_j$ is a quadratic form in at most seven variables, so it has rank at most seven. It is easily verified that each linear combination $f_i+cf_j$, with $c \ne 0$, in fact has rank eight. We can conclude that there is no linear change of variables which makes all of $f_1, \ldots, f_5$ binomials, and hence there is no basis for which $P_G$ is a binomial ideal.
\end{proof}

\section{Quadratic relations}
For any Bayesian network $G$ we have the ideal containment $I_{\glb(G)} \subseteq P_G$, and the ideal $I_{\glb(G)}$ is generated in degree two by definition. The prime ideal $P_G$ is not always generated in degree two, so the two ideals are not equal in general. It is conjectured in \cite{GSS} that the two ideals agree in degree two. 

\begin{conj}[{\cite[Conjecture 7]{GSS}}]\label{conj:global}
For any Bayesian network $G$, the quadrics of $P_G$ generate the ideal $I_{\glb(G)}$. 
\end{conj} 

We shall now prove the conjecture for the toric Bayesian networks covered by Theorem \ref{thm:toric_BN}.

\begin{thm}\label{thm:global}
Let $G$ be a Bayesian network for which the induced subgraph on all non-sinks is perfect. Then $I_{\glb(G)}$ is generated by all quadrics of $P_G$. 
\end{thm}
\begin{proof}
By Theorem \ref{thm:toric_BN} the ideal $P_G$ is toric in this case, and we use the basis for $\R[\boldsymbol{x}]$ given in Remark \ref{rmk:toric}. The idea is to prove that any quadratic binomial in $P_G$ belongs to $I_{\glb(G)}$. So, take $x_ux_{u'}-x_vx_{v'}$ where $u,u',v$ and $v'$ are $n$-vectors with entries in $\Z \cup \{+\}$ as described in Remark \ref{rmk:toric}. In particular, the $i$-th entry can only be a + if $i$ is a sink in $G$. 

The image of $x_u$ in $\R[\Theta,z]/\langle \boldsymbol{\theta} -z \rangle$ is $p=p_1 \cdots p_n$ where
\[
p_j= \begin{cases}
z & \text{if} \ u_j=+ \\
\theta(X_j=u_j \ | \ X_i=u_i \ \text{for} \ i\in \pa(j)) & \text{if} \ u_j \ne +.
\end{cases}
\]
In the same way we write the images of $x_{u'}, x_v$ and $x_{v'}$ as monomials $p'=p_1' \cdots p_n'$, $q=q_1 \cdots q_n$, and $q'=q_1' \cdots q_n'$. The equality $pp'=qq'$ implies
\begin{equation}\label{eq:proof_global}
\begin{cases}
p_{i}=q_{i} \\
p_{i}'=q_{i}'
\end{cases}
\quad \text{or} \quad 
\begin{cases}
p_{i}=q_{i}' \\
p_{i}'=q_{i}
\end{cases}
\quad \text{for each } i=1, \ldots, n.
\end{equation}
For $x_ux_{u'}-x_vx_{v'}$ to be of the form \eqref{eq:indep_eq} we can not have $u_j\ne +$ while $u_j' = +$. So suppose we are in this situation, and say $p_j=q_j$ and $p_j'=q_j'$. As $u_j \ne +$ we have $z \ne p_j=q_j$ and it follows that $u_i=v_i$ for each $i \in \pa(j)$. By \eqref{eq:proof_global} we then also have $u_i'=v_i'$  for each $i \in \pa(j)$. This allows us to replace $u_j'=+$ and $v_j'=+$ by integers $1, \ldots, \kappa_j$ and in this way produce valid relations in $P_G$. As $j$ is not a parent, only $p_j'$ and $q_j'$ are affected by this operation. Let $f_1, \ldots, f_{\kappa_j}$ denote the quadratic forms obtained from $x_ux_{u'}-x_vx_{v'}$ by changing the $j$-th entries of $u'$ and $v'$ to integers. Then 
\[
x_ux_{u'}-x_vx_{v'}=f_1 + \dots + f_{\kappa_j}.
\]  
By this argument we can now restrict to the case where for each $j$ either all of $u_j, u'_j, v_j$, and $v_j'$ are integers or all +.

Now let $A, B, C$ be subsets of $\{1, \ldots, n\}$ defined by 
\begin{align*}
C&=\{ i \ | \ u_i=u_i'=v_i=v_i' \ne +\} \\
A&=\{i \ | \ p_i=q_i \ne p_i'=q_i' \} \setminus C \\
B&=\{i \ | \ p_i=q_i' \ne p_i'=q_i \} \setminus C.
\end{align*}
Note that if $i \notin A \cup B \cup C$ then $p_i=p_i'=q_i=q_i'=z$ and in particular $i$ is a sink. With these sets $A, B, C$ our binomial relation $x_ux_{u'}-x_vx_{v'}$ is of the form \eqref{eq:indep_eq}, but we need to verify that the statement $A \indep B \ | \ C$ is true. We do this by the graph theoretical interpretation. So, assume we have a trail $\pi$ connecting vertices $k \in A$ and $\ell \in B$. We want to prove that there is a vertex $j$ on $\pi$ satisfying one of the conditions \emph{S1} or \emph{S2} given in Section \ref{subsec:ci-rel}. We may assume that the vertex next to $k$ on $\pi$ is not in $A$, because if this is the case we can consider the shorter subtrail instead. Let's first consider three special cases. 

\begin{enumerate}
\item \underline{$k \to \ell$ or $k \leftarrow \ell$ is an edge.}
Say $k \leftarrow \ell$. As $k \in A$ we have $p_k=q_k$. This requires $u_i=v_i$ for each $i \in \pa(k)$, so in particular $u_\ell = v_\ell$. As $\ell \in B$ we have $p_{\ell}=q_{\ell}'$ and $p_{\ell}'=q_{\ell}$. Altogether we then have $u_\ell = v_\ell = u'_\ell = v'_\ell$, so $\ell \in C$, contradicting $\ell \in B$. Similarly we get a contradiction if $k \to \ell$ is an edge. We can conclude that this situation does not occur. 
\item \underline{$\pi = k \to j \leftarrow \ell$ for some $j$.} As $k \in A$ we have $p_k=q_k \ne z$ and $p'_k=q_k' \ne z$. In particular $u_k=v_k$ and $u'_k=v_k'$. Then $u_k \ne v_k'$ as $k \notin C$. In the same way $ \ell \in B$ implies $u_\ell \ne v_\ell$. By \eqref{eq:proof_global} we have $p_j=q_j$ or $p_j=q_j'$. If $p_j \ne z$ then $p_j=q_j$ would imply $u_\ell=v_\ell$ and $p_j=q_j'$ would imply $u_k=v_k'$ as $k,\ell \in \pa(j)$. So $p_j=z$, which then implies $u_j=+$ and $j \notin C$. We have proved that $j$ satisfies condition \emph{S2}.
\item \underline{$\pi = k \leftarrow j \cdots$.} Since $j$ is not a sink $j \in A \cup B \cup C$. By assumption $j \notin A$, and we cannot have $j \in B$ by the same argument as in 1 above. So we can conclude $j \in C$, and therefore satisfies condition {\it S1}.
\end{enumerate}

Next, we assume $\pi = k \to j \cdots$ and argue by induction over the length of $\pi$. As we saw in the first special case $j \notin B$, and $j \notin A$ by assumption. If $j \notin C$ then $j$ is a sink, so $\pi = k \to j \leftarrow \cdots$ and condition {\it S2} is satisfied.
 Assume $j \in C$. Then {\it S1} is satisfied, unless $\pi = k \to j \leftarrow j' \cdots$. Suppose, in this case, that $j$ is a sink. As $j'$ is not a sink $j' \in A \cup B \cup C$. If $j' \in A$ we can reduce to a shorter trail, and we are done by induction. We cannot have $j' \in B$ as this would imply $j \notin C$ by 2. Hence we are left with $j' \in C$, so $j'$ satisfies {\it S1}. Finally, suppose $j$ is not a sink. Then there is an edge $k \to j'$ or $k \leftarrow j'$, as otherwise the induced subgraph on the non-sinks of $G$ would not be perfect. The case $k \leftarrow j'$ is already covered by 3 above. Assume we have the edge $k \to j'$. By induction there is a $j''$ on the shorter trail $k \to j' \cdots \ell$ satisfying {\it S1} or {\it S2}. Then $j''$ satisfies these conditions also considered as a vertex on $\pi$, except in the case of {\it S2} and $j''=j'$. But this cannot happen as $j'$ has a descendant in $C$, namely $j$. We have now proved that the statement $A \indep B \ | \ C$ holds.  
\end{proof}

It is well known that the ideal $P_G$ of a Bayesian network is quadratic if $G$ is perfect. It is proved in \cite{Lauritzen} that models given by perfect DAGs are the same as decomposable undirected graphical models. Those in turn have ideals with quadratic Gröbner basis, \cite{Hosten-Sullivant}. Alternatively, proofs for the more general result that prime ideals of balanced staged trees have quadratic Gröbner bases can be found in \cite{Ananiadi-Duarte} and \cite{GMN}. Together with Theorem \ref{thm:global} we recover Corollary \ref{cor:perfect-global}, which was originally proved in \cite{Takken} and \cite{Dobra}. See also the discussion after Theorem 4.3 in \cite{GMS}.

\begin{cor}[\cite{Dobra}, \cite{Takken}]\label{cor:perfect-global}
Let $G$ Bayesian network given by a perfect DAG. Then $I_{\glb(G)}=P_G$. 
\end{cor}

This leads to the question for which of the toric Bayesian networks in Theorem \ref{thm:toric_BN} we have $I_{\glb(G)}=P_G$. That is, when is $P_G$ quadratic? 
Adapting the terminology from \cite{GSS} we say that a DAG has an \emph{induced cycle} if there is an induced subgraph consisting of two directed paths sharing the same start and end points but are otherwise disjoint.  Among all Bayesian networks on four binary random variables that satisfy the hypothesis of Theorem \ref{thm:toric_BN}, there are two for which $I_{\glb(G)} \ne P_G$. These are networks 15 and 17 in \cite[Table 1]{GSS}, and those are precisely the two DAGs on four vertices with an induced cycle of length more than three.  In Theorem \ref{thm:quadratic} we see that such networks will always have relations of degree greater than two. 

\begin{thm}\label{thm:quadratic}
Let $G$ be a Bayesian network for which the induced subgraph on all non-sinks is perfect. If $P_G$ is quadratic then $G$ contains no induced cycle of length more than three.  
\end{thm}

In the proof of Theorem \ref{thm:quadratic} we need the following lemma. 

\begin{lem}\label{lem:degree_subgraph}
Let $G$ be a Bayesian network on vertices $1, \ldots, n$, and let $G'$ be the induced subgraph on $1, \ldots, n-1$. If a minimal generating set for $P_{G'}$ has an element of degree $d>2$, then so does any minimal generating set for $P_G$. 
\end{lem}
\begin{proof}
Let $R$ denote the polynomial ring over $\R$ on variables $x_{v_1 \ldots v_n}$ with $1 \le v_i \le \kappa_i$ for $i=1, \ldots, n$. Similarly, let $R'$ denote the polynomial ring on variables $x_{v_1 \ldots v_{n-1}}$ with $1 \le v_i \le \kappa_i$ for $i=1, \ldots, n-1$. Then we have the two maps
\[ \varphi: R \to \R[\Theta,z]/\langle \boldsymbol{\theta} -z \rangle \quad \text{and} \quad \varphi': R' \to \R[\Theta,z]/\langle \boldsymbol{\theta} -z \rangle\]
so that $P_G=\ker \varphi$ and $P_{G'}=\ker \varphi '$. Define an embedding 
\begin{align*}
\epsilon: &R' \hookrightarrow R \\
&x_{v_1 \ldots v_n} \mapsto x_{v_1 \ldots v_{n-1}+}.
\end{align*}
Then $\epsilon(P_{G'}) \subseteq P_G$. Let $\rho: \Theta \to \R$ be a map assigning values to the edge labels respecting the sum-to-one conditions as described in \eqref{eq:rho}. We define a projection 
\begin{align*}
\psi: &R \to R' \\
& x_{v_1 \ldots, v_n} \mapsto x_{v_1 \ldots v_{n-1}} \cdot \rho(\theta(X_n=v_n \ | \ X_i=v_i \ \text{for} \ i \in \pa(n))).
\end{align*}
Then $\psi(P_G) \subseteq P_{G'}$ and $\psi(\epsilon(x_{v_1 \ldots, v_{n-1}}))=x_{v_1 \ldots, v_{n-1}}$. Now, assume $P_{G'}$ has a homogeneous minimal generator $f$ of degree $d >2$. Then $\epsilon(f)$ is an element of degree $d$ in $P_G$. If $P_G$ would not have minimal generators of degree $d$, then $\epsilon(f)=h_1g_1 + \dots +h_tg_t$ for some $g_1, \ldots, g_t\in P_G$ of degrees all less than $d$. But then
\[
f=\psi(\epsilon(f)) = \psi(h_1g_1 + \dots +h_tg_t) = \psi(h_1)\psi(g_1) + \dots +\psi(h_t)\psi(g_t)
\]
would contradict $f$ being a minimal generator. Hence $P_G$ also has minimal generators of degree $d$.
\end{proof}

\begin{proof}[Proof of Theorem \ref{thm:quadratic}]
Let $G$ be a Bayesian network such that the induced subgraph on the non-sinks is perfect. In addition, assume that $G$ has a induced cycle of length more than three. We shall prove that there is a relation of degree four in $P_G$ which cannot be reduced by the quadrics of $P_G$. 

Let $\pi_1$ and $\pi_2$ be the two directed paths who constitute the induced cycle of length more than three. 
The common endpoint of $\pi_1$ and $\pi_2$ must be a sink, otherwise the induced subgraph on the non-sinks of $G$ cannot be perfect. By Remark \ref{rmk:numbering} we may number the vertices of $G$ in a suitable way, as long as the numbering respects the direction of the edges. In particular we can give the sinks consecutive numbers ending with $n$. Moreover, we can then switch the numbers of two sinks without violating the direction of the edges. If there is more than one sink we order them so that the endpoint of $\pi_1$ and $\pi_2$ is not $n$. Then we can apply Lemma \ref{lem:degree_subgraph} to remove $n$. By repeating this argument we can reduce to the case where the endpoint of $\pi_1$ and $\pi_2$ is $n$, which is the only sink. This implies that $G$ is connected, and the induced subgraph on $\{1, \ldots, n-1\}$ is perfect. 

Let $k_1$ and $k_2$ denote the parents of $n$ on $\pi_1$ and $\pi_2$. Our next step is to define disjoint vertex sets $A$, $B$, $C$ such that  
\begin{itemize}
\item $A \indep B \ | \ C$ ,
\item $A \cup B \cup C = \{1, \ldots, n-1\}$,
\item $k_1 \in A$, and the induced subgraph on $A$ is connected,
\item $k_2 \in B$, and the induced subgraph on $B$ is connected,
\item every vertex in $C$ has a child or parent in $A$, and a child or parent in $B$. 
\end{itemize}
We construct $A$, $B$, and $C$ through three steps. To start, let $A=\{k_1\}$ and $B=\{k_2\}$, and let $C$ be the set of all vertices that lies on a trail connecting $k_1$ and $k_2$, excluding $k_1$, $k_2$ and $n$. Now we claim that $A \indep B \ | \ C$. Indeed, the only way our choice of $C$ would violate the conditions {\it S1} and {\it S2} for separating $A$ and $B$ is if there would be a vertex $i$ with $k_1 \rightarrow i \leftarrow k_2$. But as the induced subgraph on $\{1, \ldots, n-1\}$ is perfect, and there is no edge between $k_1$ and $k_2$ this cannot happen. 
Next, we would like to extend $A$, $B$, and $C$ so that $A \cup B \cup C = \{1, \ldots, n-1\}$. Consider the induced subgraph on $\{1, \ldots, n-1\} \setminus C$. We extend $A$ to be all vertices in the connected component of $k_1$, and $B$ to all vertices in the connected component of $k_2$. Then we add all remaining vertices in $\{1, \ldots, n-1\}$ to  $C$. There are no new trails connecting $A$ and $B$ to consider, as such a trail would also connect $k_1$ and $k_2$. Hence the statement $A \indep B \ | \ C$ is still valid.  Last, say there is a vertex $i \in C$ with a child or parent $k \in A$. If $i$ has no child or parent in $B$ we would like to remove $i$ from $C$ and add it to $A$. If there is no trail connecting $i$ and $B$ it is clear that this can be done. Suppose there is a trail $\pi$ connecting $i$ and $B$. We can extend this trail to also include $k \in A$. Hence there is a vertex on $\pi$ satisfying {\it S1} or {\it S2}. Let $j$ be the vertex next to $i$ on the trail $\pi$. Is $j=n$ then $j$ satisfies {\it S2}, and we can safely move $i$ from $C$ to $A$. Assume $j \neq n$, and recall that $j \notin B$ as it is a child or parent of $i$. So we must have $j \in A \cup C$. If $j \in A$ there must be another vertex on $\pi$, between $j$ and the endpoint in $B$ satisfying {\it S1} or {\it S2}, and we can add $i$ to $A$. Say instead $j \in C$. Then $j$ satisfies {\it S1}, unless we have $i \to j \leftarrow \ell$ in $\pi$. Since we are considering a perfect DAG we would then have an edge between $i$ and $\ell$, and we can repeat the above argument with $\pi$ replaced by a shorter trail. We conclude that the statement $A \indep B \ | \ C$ is still valid after removing $i$ from $C$ and adding it to $A$.  In the same way we can move a vertex from $C$ to $B$. We continue doing this until every vertex that remains in $C$ has a parent or child in $A$, and a parent or child in $B$. Now all five conditions are satisfied. 

Next let's assign values to the random variables of the vertices in $A, B, C$, and encode those values in vectors $a,b,c$. In addition, we make a different assignment $a',b',c'$. We do this so that $a$ and $a'$ differ in every entry, and the same for $b$ and $b'$. We choose $c$ and $c'$ so that they have the same entry in $i$ if and only if $i \in \pa(n)$. Note that $C$ contains a vertex from $\pi_1$ or $\pi_2$, which is not a parent of $n$. Hence $c$ and $c'$ differ in at least one entry. This gives us two quadratic binomials 
\[
x_{abc}x_{a'b'c}-x_{a'bc}x_{ab'c}, \quad x_{abc'}x_{a'b'c'}-x_{a'bc'}x_{ab'c'} \  \in P_G.
\] 
Here $abc$ denotes the $n$-vector with entries determined by $a, b$, and $c$, and analogously for $a'b'c$, $a'bc$, and so on, as in \eqref{eq:indep_eq}. As $A \cup B \cup C = \{1, \ldots, n-1 \}$ the first $n-1$ entries are integers, and the last entry is +. Let $abc1$ denote the vector $abc$ with the last entry replaced by $1$, and consider the binomial 
\begin{equation}\label{eg:deg4_binom}
f=x_{abc1}x_{a'b'c1}x_{a'bc'1}x_{ab'c'1}-x_{a'bc1}x_{ab'c1}x_{abc'1}x_{a'b'c'1}.
\end{equation}
To prove that $f \in P_G$ first note that for the edge labels in the staged tree representation of $G$, we have
\[
 \theta(X_n=1 \ | \ X_i=abc_i \ \text{for} \ i\in \pa(n))= \theta(X_n=1 \ | \ X_i=abc'_i \ \text{for} \ i\in \pa(n))
\]
as $c$ and $c'$ agree on the parents of $n$. Let $\theta_{ab}$ denote this edge label. We define $\theta_{a'b}$, $\theta_{ab'}$, and $\theta_{a'b'}$ analogously. Under the map $\varphi: \R[\boldsymbol{x}] \to \R[\Theta,z]/\langle \boldsymbol{\theta} -z \rangle$ with $P_G=\ker \varphi$ we have 
\begin{align*}
&\varphi(f) =\theta_{ab}\theta_{a'b}\theta_{ab'}\theta_{a'b'}\frac{\varphi(x_{abc}x_{a'b'c}x_{a'bc'}x_{ab'c'}-x_{a'bc}x_{ab'c}x_{abc'}x_{a'b'c'})}{z^4}=0
\end{align*}
so $f \in P_G$. 

Now we shall prove that $f$ cannot be reduced by a binomial of degree two in $P_G$. As we have reduced to the case where $n$ is the only sink, and the induced subgraph on the non-sinks is perfect, $P_G$ is a binomial ideal when using the $x_u$'s with $ u=(u_1, \ldots, u_n)$ where $1 \le u_i \le \kappa_i$ for $1 \le i < n$ and $u_n=1, \ldots, \kappa_n-1$, or $+$ as basis for the polynomial ring. By Theorem \ref{thm:global} the quadrics of $P_G$ belongs to $I_{\glb(G)}$, so we ask whether $f$ can be reduced by binomials of the form \eqref{eq:indep_eq} in the $x_u$'s just described.
 If this is the case then there is a binomial $g \in I_{\glb(G)}$ such that one of the terms of $g$ divides one of the terms of $f$. This gives us 12 possible terms, one of which must occur as a term of $g$. Let's first consider the the case $g=x_{abc1}x_{a'b'c1}-x_ux_v$, for some $u,v$. 
For a conditional independence statement $D \indep E \ | \ F$ to give rise to $g$ it is necessary that $D \cup E \cup F = \{1, \ldots, n\}$, and $F \subseteq \{i \ | \ abc1_i=a'b'c1_i\}$. If there is a $j \notin F$ such that $abc1_j=a'b'c1_j$ then the statement $D \setminus \{j\} \indep E \setminus \{j\} \ | \ F \cup \{j\}$ also gives rise to the binomial $g$, but we need to verify that this conditional independence statement is true. No new trails connecting $D$ and $E$ can appear by removing $j$ from $E \cup D$ and adding $j$ to $F$. The problem that might arise is if $j$ would be the descendant of a vertex $j'$ satisfying {\it S2} on a trail $\pi$ connecting $D$ and $E$. In that case there would be an edge between the parents of $j'$ on $\pi$, as the induced subgraph on the non-sinks is perfect. This produces a shorter trail connecting $D$ and $E$, and we can find another vertex on $\pi$ satisfying {\it S1} or {\it S2}. By repeating this argument we may assume $F=\{i \ | \ abc1_i=a'b'c1_i\}$, which is precisely the set $C \cup \{n\}$. 
Then $k_1$ and $k_2$ must both be in $D$ (or both in $E$), otherwise the trail $k_1 \to n \leftarrow k_2$ violates the condition for $F$ separating $D$ and $E$. Every vertex in $A$ is connected to $k_1$ via a trail inside $A$, so $D$ must contain $A$. But in the same way every vertex of $B$ is connected to $k_2$, so $D$ must contain $B$. We end up with $D= A \cup B$ and $E= \emptyset$, and then $I_{D \indep E | F} = \langle 0 \rangle$. 

Another option is $g=x_{abc1}x_{a'bc'1}-x_ux_v$, for some $u,v$. By the same arguments as in the previous case, we are looking for a conditional independence statement $D \indep E \ | \ F$ with $D \cup E \cup F = \{1, \ldots, n\}$ and
\[
F=\{i \ | \ abc1_i=a'bc'1_i\} = B \cup (C \cap \pa(n)) \cup \{n\}.
\]
Then $D \cup E = A \cup (C \setminus \pa(n))$, and the induced subgraph on these vertices is connected. Again we are forced to choose $D= \emptyset$ or $E=\emptyset$, as there are no two proper subsets being separated by $F$. We get $I_{D \indep E | F} = \langle 0 \rangle$ also in this case. 

By symmetry, all other possibilities for terms of $g$ leads to the same conclusion. This proves that $f$ is not generated by binomials of degree two in $P_G$. 
\end{proof}

Note that both Theorem \ref{thm:global} and Theorem \ref{thm:quadratic} connects properties of the ideals $P_G$ and $I_{\glb(G)}$ with properties of the DAG $G$, not taking into account the numbers $\kappa_1, \ldots, \kappa_n$ of possible values of the random variables $X_1, \ldots, X_n$. Can the quadratic toric ideals $P_G$ from Theorem \ref{thm:toric_BN} be characterized by conditions of the DAG $G$? We conclude this section with an example showing that the numbers $\kappa_1, \ldots, \kappa_n$ do play a role when considering Bayesian networks that are not necessarily toric.  

\begin{figure}
\centering
\includegraphics[scale=1.3]{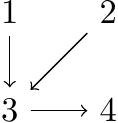}
\hspace{2cm}
\includegraphics[scale=0.9]{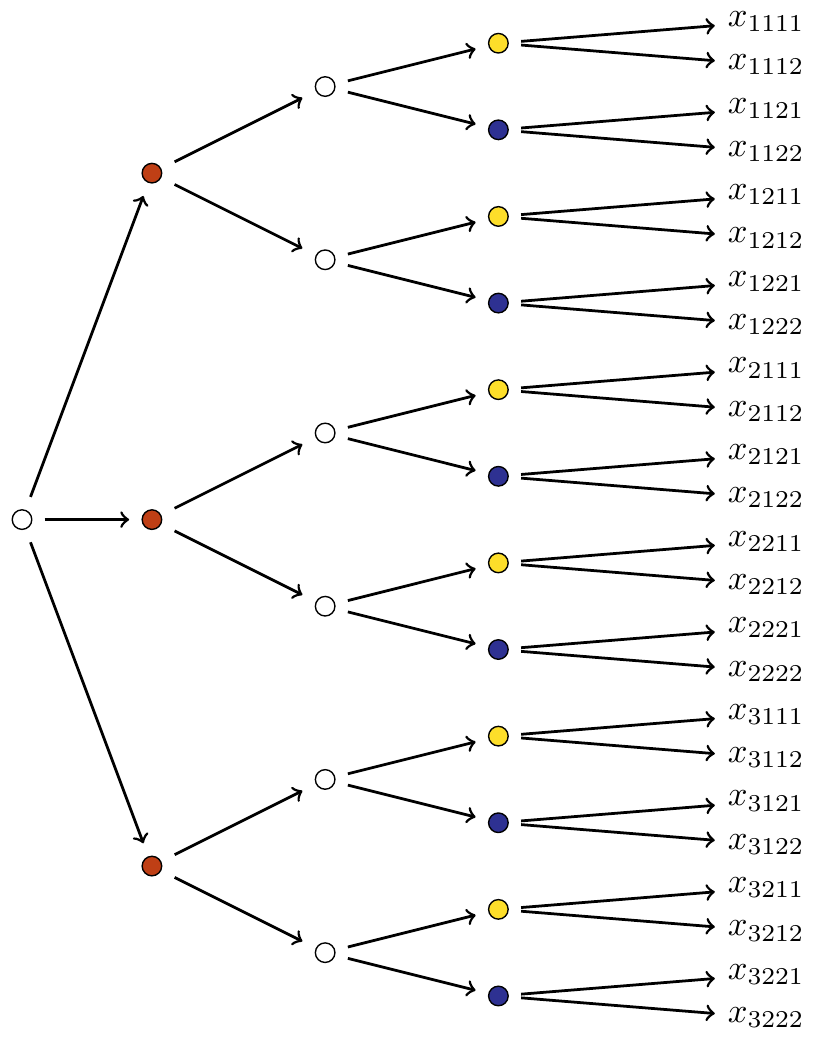}
\caption{A Bayesian net $G$ for which $P_G \ne I_{\glb(G)}$}
\label{fig:non-global}
\end{figure}

\begin{ex}
The Bayesian network $H$ on four binary random variables with the DAG in Figure \ref{fig:non-global} has $P_H=I_{\glb(H)}$. This is number 21 in \cite[Table 1]{GSS}. If we instead consider the Bayesian network $G$ on the same DAG, but where $X_1$ takes values $1,2,3$, as illustrated by the staged tree in Figure \ref{fig:non-global}, then $P_G \neq I_{\glb(G)}$. To see this, let's have a closer look at the generating set of $I_{\glb (G)}$. We have $I_{\glb{G}}=I_{2 \indep 1}+I_{4 \indep \{1,2\} | 3}$ where $I_{2 \indep 1}$ is generated by the $2 \times 2$ minors of the matrix 
\[
\begin{pmatrix}
x_{11++} & x_{21++} & x_{31++} \\
x_{12++} & x_{22++} & x_{32++}
\end{pmatrix}
\]
and $I_{4 \indep \{1,2\} | 3}$ is generated by the $2 \times 2$ minors of the two matrices
\[
\begin{pmatrix}
x_{1111} & x_{1211} & x_{2111} & x_{2211} & x_{3111} & x_{3211} \\
x_{1112} & x_{1212} & x_{2112} & x_{2212} & x_{3112} & x_{3212} 
\end{pmatrix},
\]\[
\begin{pmatrix}
x_{1121} & x_{1221} & x_{2121} & x_{2221} & x_{3121} & x_{3221} \\
x_{1122} & x_{1222} & x_{2122} & x_{2222} & x_{3122} & x_{3222} 
\end{pmatrix}.
\]
Running the Macaulay2 command {\tt trim ann(x\_1111+x\_1112)} in the quotient ring $\R[\boldsymbol{x}]/I_{\glb(G)}$ shows that $x_{111+}$ has zero divisors of degree 3. Hence we have forms $f$ of degree 3 such that $x_{111+}f \in I_{\glb(G)} \subseteq P_G$, and as $x_{111+} \notin P_G$ we have $f \in P_G$. As $f \notin I_{\glb(G)}$ it follows that $P_G \ne I_{\glb(G)}$. This means that $P_G$ is not quadratic, or Conjecture \ref{conj:global} is wrong. 

Hoping to shed some light on the nature of generating sets of ideals $P_G$ which are not necessarily toric, we here give a concrete construction of the forms $f$, deduced by analyzing the output of the Macaulay2 computation. Let  
\[
M=\begin{pmatrix}
x_{111+} & x_{211+} & x_{311+} \\
x_{111+}+x_{112+} & x_{211+}+x_{212+} & x_{311+}+x_{312+} \\
x_{121+}+x_{122+} & x_{221+}+x_{222+} & x_{321+}+x_{322+} 
\end{pmatrix}.
\]
Note that rows 2 and 3 are the matrix defining $I_{2 \indep 1}$, and hence $\det(M) \in I_{\glb(G)}$. If we were to delete all terms $x_{ij2+}$ in $M$ the resulting matrix has determinant zero, as rows 1 and 2 becomes identical. Hence, when expanding $\det(M)$ as a polynomial in $x_{ijk+}$ the terms only supported in variables $x_{ij1+}$ cancel. Hence we can write $\det(M)$ as a linear combination of terms $x_ux_vx_w$ where one of $u,v,w$ has its third entry equal to 1, and another one has its third entry equal to 2. Next, we obtain a form $f$ from $\det(M)$ by replacing each term $x_{ij1+}x_{k\ell2+}x_{mno+}$   by $x_{ij11}x_{k\ell21}x_{mno+}$. Considering the map $\varphi: \R[\boldsymbol{x}] \to \R[\Theta,z]/\langle \boldsymbol{\theta} -z \rangle$ with $P_G = \ker \varphi$ we have
\[
\varphi(x_{ij11}x_{k \ell 21})=\theta(X_4=1 \ | \ X_3=1)\theta(X_4=1 \ | \ X_3=2) \frac{\varphi(x_{ij1+}x_{k \ell 2+})}{z^2}.
\]
Applying this to every term of $f$ we get 
\[
\varphi(f)=\theta(X_4=1 \ | \ X_3=1)\theta(X_4=1 \ | \ X_3=2) \frac{\varphi(\det(M))}{z^2}=0. \qedhere
\]
\end{ex}

To the best of the authors knowledge, there is no known example of a Bayesian net without induced cycles of length $\ge 3$, and where the induced subgraph of the non-sinks is perfect, but the ideal $P_G$ is not quadratic.  

\section{Open problems}
In this last section we suggest three open problems on prime ideals of Bayesian networks. The first two questions asks whether Theorem \ref{thm:toric_BN} and Theorem \ref{thm:quadratic} gives complete characterizations of toric Bayesian networks, and toric Bayesian networks  for which $P_G=I_{\glb(G)}$. 

\begin{que}
Are all toric Bayesian networks characterized by DAGs such that the induced subgraphs on the non-sinks are perfect?
\end{que}

\begin{que}
Let $G$ be a Bayesian network such that the induced subgraph on the non-sinks is perfect. Is $P_{G}=I_{\glb(G)}$ if and only if $G$ contains no induced cycle of length greater than three?
\end{que}

If $P_G$ is toric then the algebra $\R[\boldsymbol{x}]/P_G$ is isomorphic to a monomial algebra. When $G$ is perfect we know that $P_G$ is toric, and by \cite[Theorem 3.10]{GMN} the algebra $\R[\boldsymbol{x}]/P_G$ is normal and Cohen-Macaulay. Computation shows that the same holds for every Bayesian network $G$ on four binary random variables. In the case $G$ is toric the software Normaliz \cite{Normaliz} was used to check whether the monomial parameterization defines a normal algebra. In those cases where we do not have a monomial parameterization an intermediate step is needed, as described in Example \ref{ex:normal}. 

\begin{ex}\label{ex:normal}
Let $G$ be the Bayesian network in Figure \ref{fig:non-toric}. We can define $P_G$ as the kernel of the map $\R[\boldsymbol{x}] \to \R[\Theta,z]$ defined by $x_{ijk\ell} \mapsto f_{ijk\ell}$ where
\[
\begin{aligned}
f_{1111}&=\theta_1\theta_2\theta_{31}\theta_{41},\\
f_{1112}&=\theta_1\theta_2\theta_{31}(z-\theta_{41}),\\
f_{1121}&=\theta_1\theta_2(z-\theta_{31})\theta_{42},\\
f_{1122}&=\theta_1\theta_2(z-\theta_{31})(z-\theta_{42}),\\
f_{1211}&=\theta_1(z-\theta_2)\theta_{32}\theta_{43},\\
f_{1212}&=\theta_1(z-\theta_2)\theta_{32}(z-\theta_{43}),\\
f_{1221}&=\theta_1(z-\theta_2)(z-\theta_{32})\theta_{44},\\
f_{1222}&=\theta_1(z-\theta_2)(z-\theta_{32})(z-\theta_{44}),
\end{aligned}
\ 
\begin{aligned}
f_{2111}&=(z-\theta_1)\theta_2\theta_{33}\theta_{41},\\
f_{2112}&=(z-\theta_1)\theta_2\theta_{33}(z-\theta_{41}),\\
f_{2121}&=(z-\theta_1)\theta_2(z-\theta_{33})\theta_{42},\\
f_{2122}&=(z-\theta_1)\theta_2(z-\theta_{33})(z-\theta_{42}),\\
f_{2211}&=(z-\theta_1)(z-\theta_2)\theta_{34}\theta_{43},\\
f_{2212}&=(z-\theta_1)(z-\theta_2)\theta_{34}(z-\theta_{43}),\\
f_{2221}&=(z-\theta_1)(z-\theta_2)(z-\theta_{34})\theta_{44},\\
f_{2222}&=(z-\theta_1)(z-\theta_2)(z-\theta_{34})(z-\theta_{44}).
\end{aligned}
\]
Then $A := \R[f_{1111}, \ldots, f_{2222}] \cong \R[\boldsymbol{x}]/P_G$. Let $\succ$ denote the degree reverse lexicographic term order on $\R[\Theta]$  with 
$
\theta_1 \succ \theta_2 \succ \theta_{31} \succ \dots \succ \theta_{34} \succ \theta_{41} \succ \dots \succ \theta_{44} \succ z. 
$
The \emph{initial algebra} $\init_{\succ}(A)$ is the monomial algebra generated by all leading terms of polynomials in $A$, w.\,r.\,t.\ the term order $\succ$. We run the commands 
\begin{verbatim}
A = flatten entries gens sagbi B
inA = A/leadTerm
N = gens normalToricRing inA
sort inA == sort N
\end{verbatim}
in Macaulay2 with the packages {\tt Normaliz} and {\tt SubalgebraBases} \cite{SubalgebraBases} loaded, where the input {\tt B} is a list of the polynomials $f_{1111}, \ldots, f_{2222}$. The first two rows computes the monomials generating the algebra $\init_{\succ}(A)$, and the third row computes the normalization of $\init_{\succ}(A)$. The last row checks that $\init_{\succ}(A)$ and its normalization in fact has the same generators, so $\init_{\succ}(A)$ is normal. By \cite[Corollary 2.3]{CHV} the original algebra $A$ is then  normal and Cohen-Macaulay. 
\end{ex}

\begin{que}
Is the ring $\R[\boldsymbol{x}]/P_G$ normal and Cohen-Macaulay for every Bayesian network $G$?
\end{que}

\subsection*{Acknowledgements}
I would like to thank Aldo Conca for our discussions about commutative algebra of Bayesian networks, and especially for suggesting the technique used in the proof of Theorem \ref{thm:non-toric}. Thanks also to Christiane Görgen for helpful comments on a draft of this manuscript. Finally, I thank the two anonymous referees for their careful reading.

\bibliographystyle{plain}
\bibliography{Bayesian_nets_references}

\end{document}